\documentclass[11pt,reqno]{amsart}

\usepackage[margin=1.15in]{geometry}
\usepackage{amsmath,amssymb,amsthm}
\usepackage{hyperref}
\hypersetup{colorlinks=true,linkcolor=blue,citecolor=blue,urlcolor=blue}

\theoremstyle{plain}
\newtheorem{theorem}{Theorem}
\newtheorem{lemma}[theorem]{Lemma}
\newtheorem{corollary}[theorem]{Corollary}

\theoremstyle{definition}
\newtheorem{fact}[theorem]{Fact}
\newtheorem{remark}[theorem]{Remark}

\numberwithin{equation}{section}

\DeclareMathOperator{\spn}{span}
\DeclareMathOperator{\conv}{conv}

\newcommand{\R}{\mathbb{R}}

\begin{document}

\title{Three Squares in a Rectangle}

\author{Haobo Yang}
\address{Independent scholar}
\email{hy2899@columbia.edu}


\begin{abstract}
For $x\ge1$, let $G_3(x)$ be the maximum sum of the side lengths of three
pairwise interior-disjoint, arbitrarily rotated squares contained in a
$1\times x$ rectangle. We determine this function exactly:
$G_3(x)=x+\tfrac12$ for $1\le x\le\tfrac32$, $G_3(x)=2$ for
$\tfrac32\le x\le2$, $G_3(x)=x$ for $2\le x\le3$, and $G_3(x)=3$ for
$x\ge3$. This completes the $n=3$ case of the rectangular square-packing
question posed by Richard Stanley in a 2021 MathOverflow comment. The values
for $\tfrac32\le x\le3$ follow from a strip theorem stating that three
squares in $[0,1]\times[0,H]$, $H\ge2$, have total side length at most $H$.
For $x\ge3$, the formula is immediate because each square has side length at
most $1$. The range $1\le x\le\tfrac32$ is handled by combining the two-square
theorem with an additional semi-perimeter estimate for a triangle whose two
nonhorizontal sides have opposite slopes. In particular, the special case
$x=2$ answers the question asked in the MathOverflow post. In connection with
Erd\H{o}s problem~\#106, we also prove that every five-square packing in the
unit square with an axis-parallel guillotine cut has total side length at
most $2$.
\end{abstract}

\maketitle
\section{Statement and context}

For a positive integer $n$, let $f(n)$ denote the maximum total side length of $n$
interior-disjoint squares, of arbitrary orientation, packed into a unit square.
Erd\H{o}s conjectured that $f(k^2+1) = k$; this is problem~\#106 on Bloom's index of
Erd\H{o}s problems~\cite{erdos106}. Erd\H{o}s and Soifer~\cite{ES} and, independently, Campbell and Staton~\cite{CS}
established the lower bound $f(k^2+2c+1) \ge k+c/k$ for integers $c$ with
$|c|<k$, and conjectured that equality holds; Praton~\cite{Praton} showed this
generalized conjecture to be equivalent to Erd\H{o}s's original conjecture, and
Singh~\cite{Singh} showed that the original conjecture is equivalent to the
convergence of $\sum_{k\ge1}\bigl(f(k^2+1)-k\bigr)$. Staton and Tyler~\cite{ST}
introduced the axis-parallel variant $g(n)$, for which Baek, Koizumi and
Ueoro~\cite{BKU} proved the conjectured formula. The present note concerns a
rectangular variant in which rotations are unrestricted. In a 2021 comment on a
MathOverflow question, Richard Stanley asked more generally for the maximum sum
of the side lengths of $n$ squares in a $1\times x$ rectangle, for any real
$x>1$~\cite{MO396776}. We determine this quantity exactly for $n=3$ and every
$x\ge1$; see Theorem~\ref{thm:G3}. The main geometric ingredient is the following
strip theorem.

\begin{theorem}[Strip Theorem]\label{thm:strip}
Let $H \ge 2$ and let three squares with pairwise disjoint interiors lie inside the
rectangle $R = [0,1] \times [0,H]$. The squares may be arbitrarily rotated. Then the
sum of their side lengths is at most $H$.
\end{theorem}

For $H \ge 3$ this is immediate, since each side is at most $1$ (see Step~1 below), so
the content of the theorem is the range $2 \le H \le 3$. The case $H = 2$ is exactly the
question asked on MathOverflow in July 2021~\cite{MO396776}: three
interior-disjoint squares in a $1 \times 2$ rectangle have total side length
at most $2$.

\subsection*{Relation to the existing partial answers}
A partial answer to the MathOverflow question introduces two separating
lines and bounds each square by the largest square inscribed in a triangle
using P\'olya's formula, but its monotonicity analysis is completed only under
an unproven auxiliary assumption $a,d<3.66$ on the line
parameters~\cite{MO396776}. The proof below keeps the same separating-line
setup and replaces the inscribed-square analysis by elementary semi-perimeter
lemmas for enclosing rectangles (Section~\ref{sec:lemmas}). This removes all
restrictions on the separating-line slopes and yields the full $1\times H$
statement.

\subsection*{Sharpness}
The threshold $H=2$ is sharp: for every $1\le H<2$, the constructions in
Remark~\ref{rem:sharp} have total side length greater than $H$.

\subsection*{Connection with Erd\H{o}s problem~\#106}
Corollary~\ref{cor:guillotine} shows that every five-square packing in the
unit square with an axis-parallel guillotine cut has total side length at
most $2$.

\subsection*{Notation}
A square of side $s$ with center $(x,y)$ and orientation $\theta$ has perpendicular
edge directions $u = (\cos\theta, \sin\theta)$ and $v = (-\sin\theta, \cos\theta)$.
Write
\[
  \rho = |\cos\theta| + |\sin\theta| \in [1,\sqrt2],
  \qquad
  w = s\rho .
\]
Here $w$ is the side of the axis-aligned bounding box (both horizontal and vertical
sides equal $w$), and $\rho \in [1,\sqrt2]$ measures how much wider than $s$ the
square becomes when rotated. For $\alpha \in \R$ set
\[
  \psi(\alpha) = |\alpha\cos\theta + \sin\theta| + |-\alpha\sin\theta + \cos\theta| ,
\]
so that the width of the square in the unit direction $(\alpha,1)/\sqrt{1+\alpha^2}$
equals $s\,\psi(\alpha)/\sqrt{1+\alpha^2}$, and its half-width is
$s\,\psi(\alpha)/\bigl(2\sqrt{1+\alpha^2}\bigr)$ (the support function of a square);
note $\psi(0) = \rho$.

Throughout, the \emph{semi-perimeter} of a $p \times q$ rectangle is $p + q$. We use
repeatedly the elementary identity
\begin{equation}\label{eq:span}
  \spn\{0,\; p,\; p+q\} = \tfrac12\bigl(|p| + |q| + |p+q|\bigr),
\end{equation}
where $\spn$ of a finite set of reals is its maximum minus its minimum, together with
the fact that if $p + q + r = 0$ then
$p_+ + q_+ + r_+ = \tfrac12(|p| + |q| + |r|)$, where $x_+ = \max(x,0)$ denotes the
positive part.

\section{Three semi-perimeter lemmas}\label{sec:lemmas}

\begin{fact}[flush enclosing rectangle]\label{fact:flush}
Among all rectangles, of any orientation, containing a given convex polygon, one of
minimum perimeter has a side collinear with an edge of the polygon.
\end{fact}

\begin{proof}
This is the perimeter counterpart of the minimum-area statement obtained by rotating
calipers; compare~\cite{Toussaint83}. For an angle $\varphi$, let $\sigma(\varphi)$ be
the semi-perimeter of the bounding box of the polygon in the frame rotated by
$\varphi$. Every enclosing rectangle at angle $\varphi$ contains that bounding box,
hence has semi-perimeter at least $\sigma(\varphi)$; it therefore suffices to locate the
minimum of $\sigma$, which exists since $\sigma$ is continuous and $\pi/2$-periodic. By
definition $\sigma(\varphi)$ is the sum of the lengths of the orthogonal projections of
the polygon onto $e_1(\varphi) = (\cos\varphi, \sin\varphi)$ and
$e_2(\varphi) = (-\sin\varphi, \cos\varphi)$. Each such length equals
$(z - z') \cdot e_j(\varphi)$ for the pair $z, z'$ of vertices supporting the polygon in
that direction; hence on any arc of angles on which these supporting vertices do not
change, both lengths are linear in $(\cos\varphi, \sin\varphi)$, so
$\sigma(\varphi) = A\cos\varphi + B\sin\varphi$ and $\sigma'' = -\sigma < 0$ there.
Thus $\sigma$ is strictly concave on such an arc and any interior critical point is a
strict local maximum; the minimum is therefore attained at an angle at which a
supporting vertex changes. At such an angle a supporting line contains an edge of the
polygon, so the corresponding side of the bounding box is collinear with that edge.
\end{proof}

\begin{lemma}\label{lem:right}
Every rectangle containing a right triangle with legs $\lambda, \mu \ge 0$ has
semi-perimeter at least $\lambda + \mu$.
\end{lemma}

\begin{proof}
If $\lambda\mu = 0$ the triangle degenerates to a segment of length $\lambda + \mu$, and
any rectangle containing it has semi-perimeter at least its diagonal, hence at least
$\lambda + \mu$; assume therefore $\lambda, \mu > 0$. By Fact~\ref{fact:flush} it
suffices to check the orientations flush with an edge.
Flush with either leg, the bounding box is $\lambda \times \mu$, of semi-perimeter
$\lambda + \mu$. Flush with the hypotenuse, of length
$r = \sqrt{\lambda^2 + \mu^2}$, the box is $r \times \frac{\lambda\mu}{r}$ (base times
altitude), and
\[
  r + \frac{\lambda\mu}{r} - (\lambda+\mu)
  = \frac{r^2 + \lambda\mu - r\lambda - r\mu}{r}
  = \frac{(r-\lambda)(r-\mu)}{r} \; \ge \; 0 ,
\]
since $r \ge \max(\lambda,\mu)$.
\end{proof}

\begin{lemma}\label{lem:tri}
For $a, b \in \R$ let $T(a,b) = \conv\{(0,0), (a,1), (b,1)\}$ and let
$C = \max(0,a,b) - \min(0,a,b)$ be its horizontal span, so that
$C = \tfrac12\bigl(|a| + |a-b| + |b|\bigr)$ by~\eqref{eq:span}. Then every rectangle
containing $T(a,b)$ has semi-perimeter at least
\[
  1 + C - |a| ,
  \qquad\text{and, symmetrically,}\qquad
  1 + C - |b| .
\]
\end{lemma}

\begin{proof}
We prove the first bound; the second follows by symmetry, swapping $a$ and $b$. Set $L = C - |a|$. One checks directly that $L$ is the distance
from $b$ to the interval $[\min(0,a), \max(0,a)]$, so $L \ge 0$.

\smallskip
\noindent\textbf{Case 1: $L = 0$.}
Here $b$ lies in the interval $[\min(0,a),\max(0,a)]$, so the bound reduces to
$1 + C - |a| = 1 + L = 1$ and a crude estimate suffices. The triangle has one vertex at
height $0$ and another at height $1$, so it contains two points whose distance is at
least $1$. Any $p \times q$ rectangle containing these two points has diagonal at
least their distance, so $\sqrt{p^2+q^2} \ge 1$. Since $p + q \ge \sqrt{p^2+q^2}$, we
conclude $p + q \ge 1 = 1 + L$.

\smallskip
\noindent\textbf{Case 2: $b$ on the opposite side of $a$.}
By the reflection $x \mapsto -x$ if needed we may assume $b < 0 \le a$, so $L = |b|$.
Since $b < 0 \le a$, the point $(0,1)$ lies on the top edge between $(b,1)$ and
$(a,1)$, so $T(a,b)$ contains the right triangle
$\Delta = \conv\{(0,0),(0,1),(b,1)\}$, whose legs are $1$ and $L$. By
Lemma~\ref{lem:right}, every rectangle containing $\Delta$ has semi-perimeter at least
$1 + L$; since $T(a,b) \supseteq \Delta$, the same bound holds for every rectangle
containing $T(a,b)$.

\smallskip
\noindent\textbf{Case 3: $b$ on the same side as $a$.}
After a reflection we may assume $0 \le a < b$, so $L = b - a$. Write $O = (0,0)$,
$P = (a,1)$, $Q = (b,1)$. By Fact~\ref{fact:flush} we check the three flush
orientations.

\smallskip
\emph{Flush with $OP$.} Let $\ell = \sqrt{a^2+1} = |OP|$, and project onto the unit
vector $(a,1)/\ell$ along $OP$. The images of $O = (0,0)$, $P = (a,1)$, $Q = (b,1)$ are
\[
  0, \qquad \frac{a^2+1}{\ell} = \ell, \qquad \frac{ab+1}{\ell},
\]
respectively. Since $0 \le a < b$ we have $ab \ge a^2$, hence
$\frac{ab+1}{\ell} \ge \ell$, so the images run from $0$ up to $\frac{ab+1}{\ell}$.
Using $ab + 1 = (a^2+1) + a(b-a) = \ell^2 + aL$, the span along $OP$ is
\[
  \frac{ab+1}{\ell} = \frac{\ell^2 + aL}{\ell} = \ell + \frac{aL}{\ell}.
\]
Perpendicular to $OP$, projecting onto $(-1,a)/\ell$, the images of $O, P, Q$ are
$0$, $0$, $-L/\ell$; the span (maximum minus minimum) is therefore $L/\ell$. Hence the
semi-perimeter is
\[
  \ell + \frac{aL}{\ell} + \frac{L}{\ell} = \ell + \frac{L(a+1)}{\ell} \; \ge \; 1 + L ,
\]
because $\ell \ge 1$ and $a + 1 \ge \sqrt{a^2+1} = \ell$.

\smallskip
\emph{Flush with $OQ$.} Let $m = \sqrt{b^2+1} = |OQ|$, and project onto the unit
vector $(b,1)/m$ along $OQ$. The images of $O, P, Q$ are $0$, $(ab+1)/m$, $m$, with
$0 \le (ab+1)/m \le m$ (the upper bound since $ab+1 \le b^2+1 = m^2$, using
$a \le b$), so the span along $OQ$ is $m$. Perpendicular to $OQ$, projecting onto
$(-1,b)/m$, the images of $O, P, Q$ are $0$, $L/m$, $0$, so that span is $L/m$.
Hence the semi-perimeter is $m + \dfrac{L}{m}$. Since $m = \sqrt{b^2+1} \ge 1$ and
$m \ge b \ge L$, both factors in
\[
  m + \frac{L}{m} - (1+L) = \frac{m^2 + L - m - mL}{m} = \frac{(m-1)(m-L)}{m}
\]
are nonnegative, so $m + \dfrac{L}{m} \ge 1 + L$.

\smallskip
\emph{Flush with $PQ$ (the horizontal top edge).} The bounding box is
$[0,b] \times [0,1]$, of semi-perimeter $b + 1 \ge L + 1$, since $b \ge b - a = L$.
\end{proof}

\begin{lemma}\label{lem:opposite}
Let $p,q\ge1$ and let
\[
  T_{p,q}=\conv\{(0,0),(-p,1),(q,1)\}.
\]
Every rectangle containing $T_{p,q}$ has semi-perimeter at least
\[
  p+q+\tfrac12.
\]
\end{lemma}

\begin{proof}
Put $C=p+q$. By Fact~\ref{fact:flush}, it is enough to check the three
orientations flush with an edge of the triangle.

Flush with the horizontal edge, the bounding box has semi-perimeter $C+1$.
For the edge from $(0,0)$ to $(-p,1)$, put $r=\sqrt{p^2+1}$. Projection onto
$(-p,1)/r$ gives the three values
\[
  0,\qquad r,\qquad \frac{1-pq}{r},
\]
and projection onto the perpendicular unit vector $(1,p)/r$ gives
$0,0,C/r$. Since $pq\ge1$, the two projection spans are $pC/r$ and $C/r$;
thus the semi-perimeter of this bounding box is
\[
  \frac{C(p+1)}{r}.
\]
Now $C\ge p+1\ge r$, while
\[
  r-p=\frac{1}{r+p}\le\frac12,
\]
so
\[
  \frac{C(p+1)}{r}-C
  =\frac{C}{r}(p+1-r)\ge\frac12.
\]
The orientation flush with the edge from $(0,0)$ to $(q,1)$ is symmetric.
Hence every flush bounding box, and therefore every enclosing rectangle,
has semi-perimeter at least $C+\tfrac12$.
\end{proof}

\section{Proof of the Strip Theorem}\label{sec:proof}

Let the squares be $A$, $B$, $C$, with sides $s_A, s_B, s_C$, orientations
$\theta_A, \theta_B, \theta_C$ and centers $c_i = (x_i, y_i)$, and suppose for
contradiction that
\[
  S := s_A + s_B + s_C > H .
\]
Choose $\varepsilon > 0$ with $(1-\varepsilon)S > H$, shrink each square by the factor
$1 - \varepsilon$ about its center, and relabel; all the separations below are then
strict, so we may argue with strict disjointness and no touching degeneracies.

\subsection*{Step 1: a common vertical line}
The horizontal projection of a square has length $w = s\rho \ge s$; since it must fit
in $[0,1]$, every side satisfies $s_i \le 1$. In particular $S \le 3$, so for
$H \ge 3$ we are already done; assume from now on that $2 \le H < 3$. For each pair,
\[
  s_i + s_j = S - s_k > H - 1 \; \ge \; 1 .
\]
If two squares had horizontal projections with disjoint interiors, those projections
would occupy total length $s_i\rho_i + s_j\rho_j \le 1$, forcing $s_i + s_j \le 1$, a
contradiction. So the three horizontal projections pairwise overlap; intervals have the
Helly property in dimension one, hence there is a common abscissa $x_0$: the vertical
line $x = x_0$ meets all three squares. Their intersections with this line are disjoint
segments; order the squares along it as $A$ (bottom), $B$ (middle), $C$ (top).

\subsection*{Step 2: two separating lines}
Disjoint compact convex sets can be separated by a line. A line separating $A$ from $B$
cannot be vertical, since their horizontal projections overlap in an open set, so its
normal can be scaled to $(\alpha_1, 1)$; evaluating on the common vertical line shows
that $A$ lies on the side where $\alpha_1 x + y$ is smaller. Likewise a line with
normal $(\alpha_2, 1)$ separates $B$ (below) from $C$ (above).

Consider the pair $A, B$. Projecting onto the unit normal
$n_1 = (\alpha_1,1)/\sqrt{1+\alpha_1^2}$, the two squares occupy disjoint intervals,
so the gap between their centers is at least the sum of their half-widths in this
direction:
\[
  n_1 \cdot (c_B - c_A) \; \ge \;
  \frac{s_A\,\psi_A(\alpha_1) + s_B\,\psi_B(\alpha_1)}{2\sqrt{1+\alpha_1^2}},
\]
the right-hand side being the sum of the half-widths defined in the Notation. Since
$n_1 \cdot (c_B - c_A) = \bigl(\alpha_1(x_B - x_A) + (y_B - y_A)\bigr)/\sqrt{1+\alpha_1^2}$,
multiplying through by $\sqrt{1+\alpha_1^2}$ gives
\begin{align}
  y_B - y_A &\ \ge\ \tfrac{s_A}{2}\psi_A(\alpha_1) + \tfrac{s_B}{2}\psi_B(\alpha_1)
              - \alpha_1(x_B - x_A), \label{eq:sep1}\\
  \intertext{and the same argument applied to $B, C$ with normal
  $n_2 = (\alpha_2,1)/\sqrt{1+\alpha_2^2}$ gives}
  y_C - y_B &\ \ge\ \tfrac{s_B}{2}\psi_B(\alpha_2) + \tfrac{s_C}{2}\psi_C(\alpha_2)
              - \alpha_2(x_C - x_B). \label{eq:sep2}
\end{align}

\subsection*{Step 3: the master inequality}
Adding \eqref{eq:sep1} and \eqref{eq:sep2}, the left-hand sides telescope to
$y_C - y_A$:
\begin{equation}\label{eq:added}
  y_C - y_A \;\ge\; \tfrac{s_A}{2}\psi_A(\alpha_1)
    + \tfrac{s_B}{2}\bigl(\psi_B(\alpha_1) + \psi_B(\alpha_2)\bigr)
    + \tfrac{s_C}{2}\psi_C(\alpha_2) - X ,
\end{equation}
where
\[
  X := \alpha_1(x_B - x_A) + \alpha_2(x_C - x_B)
     = -\alpha_1 x_A + (\alpha_1 - \alpha_2)x_B + \alpha_2 x_C .
\]
Since each square lies in $R$, its center keeps half a bounding box away from every
side: $\tfrac{w_i}{2} \le x_i \le 1 - \tfrac{w_i}{2}$ and
$\tfrac{w_i}{2} \le y_i \le H - \tfrac{w_i}{2}$. The vertical bounds give
$y_C - y_A \le H - \tfrac12(w_A + w_C)$, so \eqref{eq:added} becomes
\begin{equation}\label{eq:beforeX}
  \tfrac{s_A}{2}\psi_A(\alpha_1)
    + \tfrac{s_B}{2}\bigl(\psi_B(\alpha_1) + \psi_B(\alpha_2)\bigr)
    + \tfrac{s_C}{2}\psi_C(\alpha_2)
    + \tfrac12\bigl(w_A + w_C\bigr) \;\le\; H + X .
\end{equation}
For $X$ we use the horizontal bounds. Each center satisfies
$\tfrac{w_i}{2} \le x_i \le 1 - \tfrac{w_i}{2}$, so for any $\gamma \in \R$ the linear
form $\gamma x_i$ attains its maximum at one of the two endpoints, and in either case
\[
  \gamma x_i \; \le \; \gamma_+ - \frac{|\gamma| w_i}{2} .
\]
We apply this to the three terms of $X$, namely to $-\alpha_1 x_A$, to
$(\alpha_1 - \alpha_2) x_B$ and to $\alpha_2 x_C$. Their coefficients
$-\alpha_1$, $\alpha_1 - \alpha_2$, $\alpha_2$ sum to zero, so by the identity in the
Notation the sum of their positive parts equals
$\tfrac12\bigl(|\alpha_1| + |\alpha_1 - \alpha_2| + |\alpha_2|\bigr)$; write $C_\alpha$
for this quantity. Adding the three bounds therefore gives
\begin{equation}\label{eq:Xbound}
  X \; \le \; C_\alpha
    - \tfrac12\bigl(|\alpha_1| w_A + |\alpha_1 - \alpha_2| w_B + |\alpha_2| w_C\bigr) .
\end{equation}
Substituting \eqref{eq:Xbound} into \eqref{eq:beforeX} and moving every $w_i$ term to
the left-hand side, the coefficient of $s_i$ there becomes, after $w_i = s_i\rho_i$,
\[
  \begin{aligned}
    K_A &= \tfrac12\bigl(\psi_A(\alpha_1) + \rho_A(1 + |\alpha_1|)\bigr), \\
    K_B &= \tfrac12\bigl(\psi_B(\alpha_1) + \psi_B(\alpha_2)
           + \rho_B|\alpha_1 - \alpha_2|\bigr), \\
    K_C &= \tfrac12\bigl(\psi_C(\alpha_2) + \rho_C(1 + |\alpha_2|)\bigr),
  \end{aligned}
\]
and we obtain the \emph{master inequality}
\begin{equation}\label{eq:master}
  K_A s_A + K_B s_B + K_C s_C \; \le\; H + C_\alpha .
\end{equation}

\subsection*{Step 4: the coefficients are semi-perimeters}
At this point the coefficients $K_i$ are purely algebraic, and nothing in their
definition suggests a lower bound. The observation that drives the proof is that each of
them is \emph{exactly} the semi-perimeter of a bounding box of a triangle built from
$\alpha_1, \alpha_2$, read in the frame of the corresponding square; Lemmas~\ref{lem:right} and~\ref{lem:tri} then bound them from below.

For a bounded set $V \subset \R^2$, the bounding box of $V$ in the frame of square $i$,
with axes $u_i, v_i$, has semi-perimeter $\spn(u_i \cdot V) + \spn(v_i \cdot V)$.
Applying~\eqref{eq:span} to each of the two spans, one finds:
\begin{itemize}
  \item $K_A$ is the semi-perimeter of the bounding box, in the frame of $A$, of the
        right triangle
        $\Delta_{\alpha_1} = \conv\{(0,0), (\alpha_1,0), (\alpha_1,1)\}$, whose legs
        are $|\alpha_1|$ and $1$;
  \item $K_B$ is the semi-perimeter of the bounding box, in the frame of $B$, of the
        triangle $T(\alpha_1, \alpha_2)$ of Lemma~\ref{lem:tri}, whose horizontal span
        is exactly $C_\alpha$;
  \item $K_C$ is the semi-perimeter of the bounding box, in the frame of $C$, of
        $\Delta_{\alpha_2}$, whose legs are $|\alpha_2|$ and $1$.
\end{itemize}

We carry out the computation for $K_A$; the other two are of the same shape. The
vertices of $\Delta_{\alpha_1}$ have $u_A$-images $0$, $\alpha_1\cos\theta_A$,
$\alpha_1\cos\theta_A + \sin\theta_A$ and $v_A$-images $0$, $-\alpha_1\sin\theta_A$,
$-\alpha_1\sin\theta_A + \cos\theta_A$; both triples are of the form $\{0, p, p+q\}$,
so \eqref{eq:span} applies to each and the two spans add up to
\begin{align*}
&\spn(u_A \cdot \Delta_{\alpha_1})
 + \spn(v_A \cdot \Delta_{\alpha_1}) \\[2pt]
&= \tfrac12\Bigl(
 |\alpha_1\cos\theta_A|
 +|\alpha_1\sin\theta_A|
 +|\cos\theta_A|
 +|\sin\theta_A|
 \Bigr)\\
&\qquad
 +\tfrac12\Bigl(
 |\alpha_1\cos\theta_A+\sin\theta_A|
 +|-\alpha_1\sin\theta_A+\cos\theta_A|
 \Bigr)\\[2pt]
&= \tfrac12\Bigl(
 |\alpha_1|\rho_A+\rho_A+\psi_A(\alpha_1)
 \Bigr)\\
&= \tfrac12\bigl(
 \psi_A(\alpha_1)+\rho_A(1+|\alpha_1|)
 \bigr)
 = K_A .
\end{align*}
The six absolute values assemble exactly into $\rho_A$, $|\alpha_1|\rho_A$ and
$\psi_A(\alpha_1)$. Each of these bounding boxes is a rectangle containing the
corresponding triangle, so Lemmas~\ref{lem:right} and~\ref{lem:tri} give
\begin{align}
  K_A &\ \ge\ 1 + |\alpha_1|, \label{eq:KA}\\
  K_C &\ \ge\ 1 + |\alpha_2|, \label{eq:KC}\\
  K_B &\ \ge\ 1 + C_\alpha - |\alpha_1|, \label{eq:KB1}\\
  K_B &\ \ge\ 1 + C_\alpha - |\alpha_2|. \label{eq:KB2}
\end{align}

\subsection*{Step 5: pairwise lower bounds}
Let $M = 2 + C_\alpha$. Then \eqref{eq:KA} $+$ \eqref{eq:KB1} and
\eqref{eq:KC} $+$ \eqref{eq:KB2} give $K_A + K_B \ge M$ and $K_B + K_C \ge M$; and
since $|\alpha_1| + |\alpha_2| \ge C_\alpha$, adding \eqref{eq:KA} and \eqref{eq:KC}
gives $K_A + K_C \ge M$ as well. Adding \eqref{eq:KC} to $K_A + K_B \ge M$,
\begin{equation}\label{eq:Ksum}
  K_\Sigma := K_A + K_B + K_C \; \ge \; M + 1 .
\end{equation}

\subsection*{Step 6: conclusion of the proof}
Put $t_i = 1 - s_i \ge 0$, so that $t_A + t_B + t_C = 3 - S < 3 - H$. From the pairwise bounds of Step~5, each $K_i \le K_\Sigma - M$ (subtract the bound on
the other two from $K_\Sigma$), and $K_\Sigma - M \ge 1 > 0$ by~\eqref{eq:Ksum}.
Hence
\begin{align*}
  \sum_i K_i s_i
    &= K_\Sigma - \sum_i K_i t_i \\
    &\ \ge\ K_\Sigma - (K_\Sigma - M)\sum_i t_i \\
    &\ >\ K_\Sigma - (K_\Sigma - M)(3-H) \\
    &\ =\ M + (H-2)(K_\Sigma - M) .
\end{align*}
Since $H - 2 \ge 0$ and $K_\Sigma - M \ge 1$, the last expression satisfies
\[
  M + (H-2)(K_\Sigma - M) \;\ge\; M + (H-2) \;=\; (2 + C_\alpha) + (H-2)
    \;=\; H + C_\alpha .
\]
Thus $\sum_i K_i s_i > H + C_\alpha$, contradicting~\eqref{eq:master} and completing
the proof.
\qed

\begin{remark}[Attainment]
For every $H \in [2,3]$ the value $H$ is attained: three squares of sides
$1$, $\tfrac{H-1}{2}$, $\tfrac{H-1}{2}$, stacked vertically, have total side length
$1 + (H-1) = H$.
\end{remark}

\begin{remark}[Sharpness of $H\ge2$]\label{rem:sharp}
The threshold is sharp throughout $1\le H<2$. For $1\le H\le\tfrac32$, the
three-square construction of sides $\tfrac12,\tfrac12,H-\tfrac12$ has total side
length $H+\tfrac12>H$. For $\tfrac32\le H<2$, a unit square together with two
squares of side $\tfrac12$, placed side by side above it, has total side length
$2>H$.
\end{remark}

\begin{remark}[Where each ingredient is used]
Convexity gives the separating lines; the square shape is used through the support
function $\psi$ and through the fact that the projection of a square has length at
least its side; the container is used through the wall inequalities. In the proof of
the Strip Theorem, the hypothesis $H\ge2$ is used only in Step~1, to force every
pairwise sum above $1$, and in Step~6, where $H-2\ge0$. Steps~2--5, including the
master inequality and the coefficient bounds, remain valid for every positive
container height once a common vertical line is available. No angle or slope is
restricted at any point.
\end{remark}

\section{Three squares in a \texorpdfstring{$1 \times x$}{1-by-x} rectangle}

Let $G_3(x)$ denote the maximum total side length of three
interior-disjoint squares, with arbitrary orientations, contained in a
$1\times x$ rectangle, where $x\ge1$.

\begin{theorem}\label{thm:G3}
For every $x\ge1$,
\[
  G_3(x)=
  \begin{cases}
    x+\tfrac12, & 1\le x\le\tfrac32,\\[2pt]
    2,           & \tfrac32\le x\le2,\\[2pt]
    x,           & 2\le x\le3,\\[2pt]
    3,           & x\ge3.
  \end{cases}
\]
\end{theorem}

\begin{proof}
The upper bounds for $x\ge\tfrac32$ are immediate from
Theorem~\ref{thm:strip}: for $\tfrac32\le x\le2$, enlarge the container to a
$1\times2$ rectangle; for $2\le x\le3$, apply the theorem directly; and for
$x\ge3$, each side is at most $1$.

It remains to prove
\[
  G_3(x)\le x+\tfrac12
  \qquad (1\le x\le\tfrac32).
\]
Suppose to the contrary that three squares have sides $s_A,s_B,s_C$ and
\[
  S:=s_A+s_B+s_C>x+\tfrac12.
\]
As in the proof of Theorem~\ref{thm:strip}, shrink the squares slightly while
preserving this strict inequality, so all separations may be taken strict.
Any two of the squares lie in an $x\times x$ square, because
$[0,1]\times[0,x]\subseteq[0,x]^2$. Theorem~\ref{thm:two}, after scaling,
therefore gives
\begin{equation}\label{eq:pair-x}
  s_i+s_j\le x
  \qquad (i\ne j).
\end{equation}
If, say, $s_A+s_B\le1$, then adding this inequality to the two relevant
instances of~\eqref{eq:pair-x} gives
\[
  2S=(s_A+s_B)+(s_A+s_C)+(s_B+s_C)\le1+2x,
\]
contrary to $S>x+\tfrac12$. Hence every pairwise sum is greater than $1$.

The horizontal projections consequently overlap pairwise in their interiors,
so they have a common abscissa. Order the squares along the corresponding
vertical line as $A$ (bottom), $B$ (middle), and $C$ (top). Steps~2--4 of the
proof of Theorem~\ref{thm:strip} apply verbatim with $H$ replaced by $x$.
Thus there are real numbers $\alpha_1,\alpha_2$ and coefficients
$K_A,K_B,K_C$ satisfying
\begin{equation}\label{eq:master-small}
  K_As_A+K_Bs_B+K_Cs_C\le x+C_\alpha,
\end{equation}
where
\[
  C_\alpha=\tfrac12\bigl(|\alpha_1|+|\alpha_1-\alpha_2|+|\alpha_2|\bigr),
\]
and~\eqref{eq:KA}--\eqref{eq:KB2} hold. Put
\[
  L_A=K_A-1,\qquad L_B=K_B-1,\qquad L_C=K_C-1.
\]
We claim that
\begin{equation}\label{eq:Lclaim}
  L_As_A+L_Bs_B+L_Cs_C\ge C_\alpha-\tfrac12.
\end{equation}

Set $p=|\alpha_1|$ and $q=|\alpha_2|$. The estimates below are symmetric in
$(p,s_A)$ and $(q,s_C)$, so assume $p\ge q$.

If $\alpha_1\alpha_2\ge0$, then $C_\alpha=p$, and
\eqref{eq:KA}--\eqref{eq:KB2} give
\[
  L_A\ge p,\qquad L_B\ge p-q,\qquad L_C\ge q.
\]
Consequently,
\[
  L_As_A+L_Bs_B+L_Cs_C
  \ge (p-q)(s_A+s_B)+q(s_A+s_C)\ge p=C_\alpha.
\]

Suppose next that $\alpha_1\alpha_2<0$. Then $C_\alpha=p+q$, and the same
coefficient bounds give
\[
  L_A\ge p,\qquad L_B\ge p,\qquad L_C\ge q.
\]
If $q\le1$, then
\begin{align*}
  L_As_A+L_Bs_B+L_Cs_C
  &\ge \left(p-\frac q2\right)(s_A+s_B)
      +\frac q2(s_A+s_C)+\frac q2(s_B+s_C)\\
  &\ge p+\frac q2
   \ge p+q-\frac12
   = C_\alpha-\frac12.
\end{align*}
If $q>1$, then after a reflection and, if necessary, interchanging $p$
and $q$, the triangle $T(\alpha_1,\alpha_2)$ associated with $K_B$ in
Step~4 is congruent to the triangle in Lemma~\ref{lem:opposite}. Hence
$K_B\ge C_\alpha+\tfrac12$, or $L_B\ge C_\alpha-\tfrac12$. Using also
$s_B\le1$ and the strict pairwise inequalities,
\begin{align*}
  L_As_A+L_Bs_B+L_Cs_C
  &\ge p s_A+\left(C_\alpha-\frac12\right)s_B+q s_C\\
  &=p(s_A+s_B)+q(s_B+s_C)-\frac12s_B\\
  &>p+q-\frac12s_B
   \ge C_\alpha-\frac12.
\end{align*}
This proves~\eqref{eq:Lclaim}. Therefore
\[
  K_As_A+K_Bs_B+K_Cs_C
  =S+L_As_A+L_Bs_B+L_Cs_C
  >x+C_\alpha,
\]
contradicting~\eqref{eq:master-small}.

The matching constructions are axis-parallel. For $1\le x\le\tfrac32$, use
sides $\tfrac12,\tfrac12,x-\tfrac12$; for
$\tfrac32\le x\le2$, use sides $1,\tfrac12,\tfrac12$; for
$2\le x\le3$, use sides $1,\tfrac{x-1}{2},\tfrac{x-1}{2}$; and for
$x\ge3$, stack three unit squares. These attain the four displayed values.
\end{proof}

This completes the $n=3$ case of the rectangular square-packing question
posed by Richard Stanley in a 2021 comment on~\cite{MO396776}.

By scaling and, if necessary, interchanging the coordinate axes,
Theorem~\ref{thm:G3} determines the maximum total side length of three
squares in every rectangle.

\begin{corollary}[Rescaled strip form]\label{rem:strip}
Let $F_3(u)$ denote the maximum total side length of three
interior-disjoint squares in $[0,u]\times[0,1]$, where $0<u\le1$.
Then
\[
  F_3(u)=u\,G_3(1/u)=
  \begin{cases}
    3u,             & 0<u\le\tfrac13,\\[2pt]
    1,              & \tfrac13\le u\le\tfrac12,\\[2pt]
    2u,             & \tfrac12\le u\le\tfrac23,\\[2pt]
    1+\dfrac{u}{2}, & \tfrac23\le u\le1.
  \end{cases}
\]
In particular, $F_3(u)\le2u$ for every $\tfrac12\le u\le1$.
\end{corollary}

\section{Axis-parallel guillotine cuts in five-square packings}

\begin{corollary}\label{cor:guillotine}
Let five interior-disjoint squares in the unit square have total side length
$S_5$. Suppose that the packing admits an axis-parallel guillotine cut,
namely, an axis-parallel line that avoids the interiors of all five squares
and separates them into two nonempty subfamilies. Then $S_5\le2$.
\end{corollary}

\begin{proof}
After rotating the entire configuration through a right angle if necessary, we may
assume that $\ell$ is vertical, say $\ell = \{x = t\}$. Reflecting in the vertical
midline if necessary, we may further assume that $k$ squares lie on the left and $5-k$
on the right, with $k \in \{1,2\}$.

\smallskip
\noindent\textbf{Case $1+4$.}
The single left square, of side $a$, lies in $[0,t] \times [0,1]$ and its horizontal
projection has length at least $a$, so $t \ge a$. The other four squares lie in the
right strip $[t,1] \times [0,1]$, of area $1 - t \le 1 - a$, and have disjoint
interiors, so their areas satisfy
\[
  \sum_{i=1}^4 b_i^2 \; \le \; 1 - a .
\]
By Cauchy--Schwarz, $\bigl(\sum b_i\bigr)^2 \le 4 \sum b_i^2$, hence
$\sum b_i \le 2\sqrt{\sum b_i^2} \le 2\sqrt{1-a}$ and
\[
  S_5 \; \le \; a + 2\sqrt{1-a} .
\]
Since $a \le 1$ we have $2 - a > 0$, so squaring is legitimate and
\[
  a + 2\sqrt{1-a} \le 2
  \iff 2\sqrt{1-a} \le 2 - a
  \iff 4(1-a) \le (2-a)^2
  \iff 0 \le a^2 ,
\]
which holds. Hence $S_5 \le 2$.

\smallskip
\noindent\textbf{Case $2+3$.}
Let the right strip have width $u$, so the left one has width $1 - u$. The two left
squares contribute at most $\min\bigl(1,\ 2(1-u)\bigr)$: each side is at most the strip
width, and the two-square theorem (Theorem~\ref{thm:two}) in the ambient unit square
caps the pair at~$1$. The three right squares, in a strip $u \times 1$, contribute at
most $3u$ for $u \le \tfrac13$, at most $1$ for $\tfrac13 \le u \le \tfrac12$, and at
most $2u$ for $u \ge \tfrac12$, by Corollary~\ref{rem:strip}. Then
\begin{align*}
  u \le \tfrac13: &\quad S_5 \le 1 + 3u \le 2; \\
  \tfrac13 \le u \le \tfrac12: &\quad S_5 \le 1 + 1 = 2; \\
  u \ge \tfrac12: &\quad S_5 \le 2(1-u) + 2u = 2. \qedhere
\end{align*}
\end{proof}

\subsection*{Consequence for Erd\H{o}s problem~\#106}
Consequently, any five-square packing in the unit square with total side
length greater than $2$ admits no axis-parallel guillotine cut. The equality
$f(5)=2$ is attributed to Newman through a personal communication to
Erd\H{o}s~\cite{erdos106}; we are not aware of a published proof.

\appendix

\section{The two-square theorem and small values}

\begin{theorem}\label{thm:two}
Two squares with disjoint interiors in the unit square, with sides $a$ and
$b$, satisfy
\[
a+b\le1.
\]
\end{theorem}

The result is attributed to Erd\H{o}s in an early article in a Hungarian
journal for secondary-school students; see the history notes
at~\cite{erdos106}. We give the following self-contained reconstruction.

\begin{lemma}[corner depth]\label{lem:corner}
A square $Q$ of side $s$ contained in the first quadrant contains a point $z$ with
$z_x \ge s$ and $z_y \ge s$; if $Q$ is not axis-parallel, both inequalities are strict.
\end{lemma}

\begin{proof}
Reducing the orientation modulo $\pi/2$, take $\theta \in [-\tfrac\pi4, \tfrac\pi4]$ and
$\rho = \cos\theta + |\sin\theta| \in [1,\sqrt2]$. The horizontal and vertical
half-widths of $Q$ are $s\rho/2$, so its center $c$ satisfies
$c_x, c_y \ge s\rho/2$. Put $z = c + \tfrac{s}{2\rho}(1,1)$. Then
\[
  |(z-c)\cdot u| = \tfrac{s}{2\rho}|\cos\theta + \sin\theta| \le \tfrac{s}{2},
  \qquad
  |(z-c)\cdot v| = \tfrac{s}{2\rho}|\cos\theta - \sin\theta| \le \tfrac{s}{2},
\]
so $z \in Q$; and
$z_x, z_y \ge \tfrac{s}{2}\bigl(\rho + \tfrac1\rho\bigr) \ge s$
by AM--GM, strictly if $\rho > 1$.
\end{proof}

\begin{proof}[Proof of Theorem~\ref{thm:two}]
The interiors are disjoint convex sets, so a line $px + qy = t$ separates the squares;
composing with the reflections $x \mapsto 1-x$ and $y \mapsto 1-y$ of the container we
may take $p, q \ge 0$ with $p + q > 0$, and the $a$-square in $px + qy \le t$.
Lemma~\ref{lem:corner} gives a point $z$ of the $a$-square with $z_x, z_y \ge a$, whence
$t \ge p z_x + q z_y \ge a(p+q)$. Reflecting the $b$-square through the center of the
container and applying the lemma gives a point $z'$ of the $b$-square with
$z'_x, z'_y \le 1 - b$, whence $t \le (1-b)(p+q)$. Dividing by $p + q$ gives
$a + b \le 1$.
\end{proof}

\begin{corollary}\label{cor:small}
$f(3) = \tfrac32$ and $f(4) = 2$.
\end{corollary}

\begin{proof}
For three squares of sides $a, b, c$, summing $a + b \le 1$ over the three pairs gives
$2(a+b+c) \le 3$; three squares of side $\tfrac12$ in an L attain it. For $f(4)$,
Cauchy--Schwarz gives $\sum s_i \le 2\sqrt{\sum s_i^2} \le 2$, attained by the
$2\times2$ grid.
\end{proof}


\begin{thebibliography}{99}

\bibitem{BKU}
J.~Baek, J.~Koizumi and T.~Ueoro,
\emph{A note on the Erd\H{o}s conjecture about square packing},
preprint, \texttt{arXiv:2411.07274} (2024).

\bibitem{CS}
C.~Campbell and W.~Staton,
\emph{A square-packing problem of Erd\H{o}s},
Amer. Math. Monthly \textbf{112} (2005), no.~2, 165--167.

\bibitem{ES}
P.~Erd\H{o}s and A.~Soifer,
\emph{Squares in a square},
Geombinatorics \textbf{4} (1995), no.~4, 110--114.

\bibitem{erdos106}
T.~F. Bloom,
\emph{Erd\H{o}s Problem \#106},
\url{https://www.erdosproblems.com/106}, accessed 2026-08-25.

\bibitem{MO396776}
\emph{Three squares in a rectangle},
MathOverflow question 396776, asked 4 July 2021,
\url{https://mathoverflow.net/questions/396776}.

\bibitem{Praton}
I.~Praton,
\emph{Packing squares in a square},
Math. Mag. \textbf{81} (2008), no.~5, 358--361.

\bibitem{Singh}
A.~R. Singh,
\emph{On a square packing conjecture of Erd\H{o}s},
preprint, \texttt{arXiv:2601.22163} (2026).

\bibitem{ST}
W.~Staton and B.~Tyler,
\emph{On the Erd\H{o}s square-packing conjecture},
Geombinatorics \textbf{17} (2007), no.~2, 88--94.

\bibitem{Toussaint83}
G.~T. Toussaint,
\emph{Solving geometric problems with the rotating calipers},
Proc. IEEE MELECON~'83, Athens, Greece, May 1983.

\end{thebibliography}
\end{document}